\documentclass{amsart}
\usepackage{graphicx}
\usepackage[utf8]{inputenc}
\newcommand{\N}{\mathbb{N}}
\newtheorem{Theo}{Theorem}
\newtheorem{Lem}{Lemma}
\newtheorem{Cor}{Corollary}
\begin{document}
\title[Nowhere differentiable functions]{A continuum dimensional algebra of nowhere differentiable functions} 
\author[J.-C. Schlage-Puchta]{Jan-Christoph Schlage-Puchta}
\begin{abstract}
We construct an algebra of dimension $2^{\aleph_0}$ consisting only of functions which in no point possess a finite one-sided derivative. We further show that some well known nowhere differentiable functions generate algebras, which contain functions which are differentiable at some points, but where for all functions in the algebra the set of points of differentiability is quite small.
\end{abstract}
\maketitle
MSC-Index 26A27, 26A16, 26A30

\section{Introduction and results}
In 1872, Weierstraß was the first to construct a continuous nowhere differentiable real function. This construction was first seen as a pathological example, it later turned out that nowhere differentiable functions are actually quite common. In fact, a standard application of the Baire category theorem is the proof that in a certain sense most functions are nowhere differentiable. Nowhere differentiable functions occur almost surely as the path of a Brownian motion. An explicit example of a nowhere differentiable function is Takagi's function \cite{Takagi}, that describes the asymptotics of the binary sum of digits. It can be defined as the unique solution $T:[0,1]\rightarrow\mathbb{R}$  of the functional equation 
\[
T(x) = \begin{cases} x+\frac{1}{2}T(2x), & 0\leq x\leq\frac{1}{2},\\
1-x+T(2x-1), & \frac{1}{2}\leq x\leq 1.
\end{cases}
\]
Fonf, Guraiy and Kadets\cite{FGK} constructed an infinite dimensional space $V\leq C^0([0,1])$, such that every non-zero element of $V$ is nowhere differentiable. Girgensohn \cite{Girgensohn} and Bobok \cite{Bobok} gave a construction of an infinite dimensional space such that every non-zero element has at no point a one-sided derivative.

In this note we construct algebras of nowhere differentiable functions of dimension $2^{\aleph_0}$. 

For some $\alpha\in(0,1]$ we say that a function $f:[0,1]\rightarrow\mathbb{R}$ is $\alpha$-Hölder at $x_0$ from the right, if 
\[
\limsup_{y\searrow x_0}\frac{f(y)-f(x_0)}{|y-x_0|^\alpha} <\infty.
\]
$\alpha$-Hölder from the left is defined analogously. We say that a function $f$ is completely non-Hölder, if there is no $x_0$ and no $\alpha>0$ such that $f$ is $\alpha$-Hölder at $x_0$ from the left or from the right.

Then we show the following.
\begin{Theo}
\label{thm:characterization}
Let $f:[0,1]\rightarrow\mathbb{R}$ be a continuous function. Then the following statements are equivalent.
\begin{enumerate}
\item $f$ is completely non-Hölder;
\item for every non-constant polynomial $P$ we have that $P(f)$ does not have a one-sided derivative equal to 0 at some point;
\item for every function $F$ which is non-constant locally analytic in some neighbourhood of $f([0,1])$ we have that $F(f)$ is completely non-Hölder.
\end{enumerate}
\end{Theo}

\begin{Cor}\label{Cor:dimension}
Let $f:[0,1]\rightarrow (0, \infty)$ be a continuous completely non-Hölder function. Then $\{f^\lambda:\lambda>0\}$ generates a $2^{\aleph_0}$-dimensional algebra $\mathcal{A}$, such that every non-zero element in $\mathcal{A}$ has nowhere a finite one-sided derivative.
\end{Cor}

Note that Bayart and Quarta \cite{BQ} already constructed an algebra of nowhere differentiable functions with countably infinite transcendence degree, clearly the algebra given in Corollary~\ref{Cor:dimension} has uncountable transcendence degree.

Next we show that there are actually functions to which we can apply Theorem~\ref{thm:characterization}. We do so in two different ways. First, we show that completely non-Hölder functions are quite common.

\begin{Theo}
\label{thm:meagre}
The set of functions which are not completely non-Hölder is meagre.
\end{Theo}

We then give explicit examples of such functions and characterize them in terms of their Faber-Schauder expansion. As an example we have the following.

\begin{Cor}
The function $f(x)=\sum_{n\geq 0} \frac{1}{(n+1)^2} \sum_{i=0}^{2^n-1} \sigma_{n,i}(x)$ is completely non-H\"older.
\end{Cor}

Here $\sigma_{n,i}$ is the Faber-Schauder basis given by
\[
\sigma_{0,0}(x)=\begin{cases}
0, & x\leq 0\mbox{ or }x\geq 1\\
2x, & 0\leq x\leq\frac{1}{2}\\
2-2x, & \frac{1}{2}\leq x\leq 1
\end{cases},
\]
and, for $n\in\N$ and $0\leq i\leq 2^n-1$, $\sigma_{n,i}=\sigma\left(2^nx-i\right)$. 

We finally give examples of functions that generate algebras of functions that are differentiable at some points, but where such points are quite rare.

\begin{figure}
\includegraphics[width=\textwidth, height = 5cm]{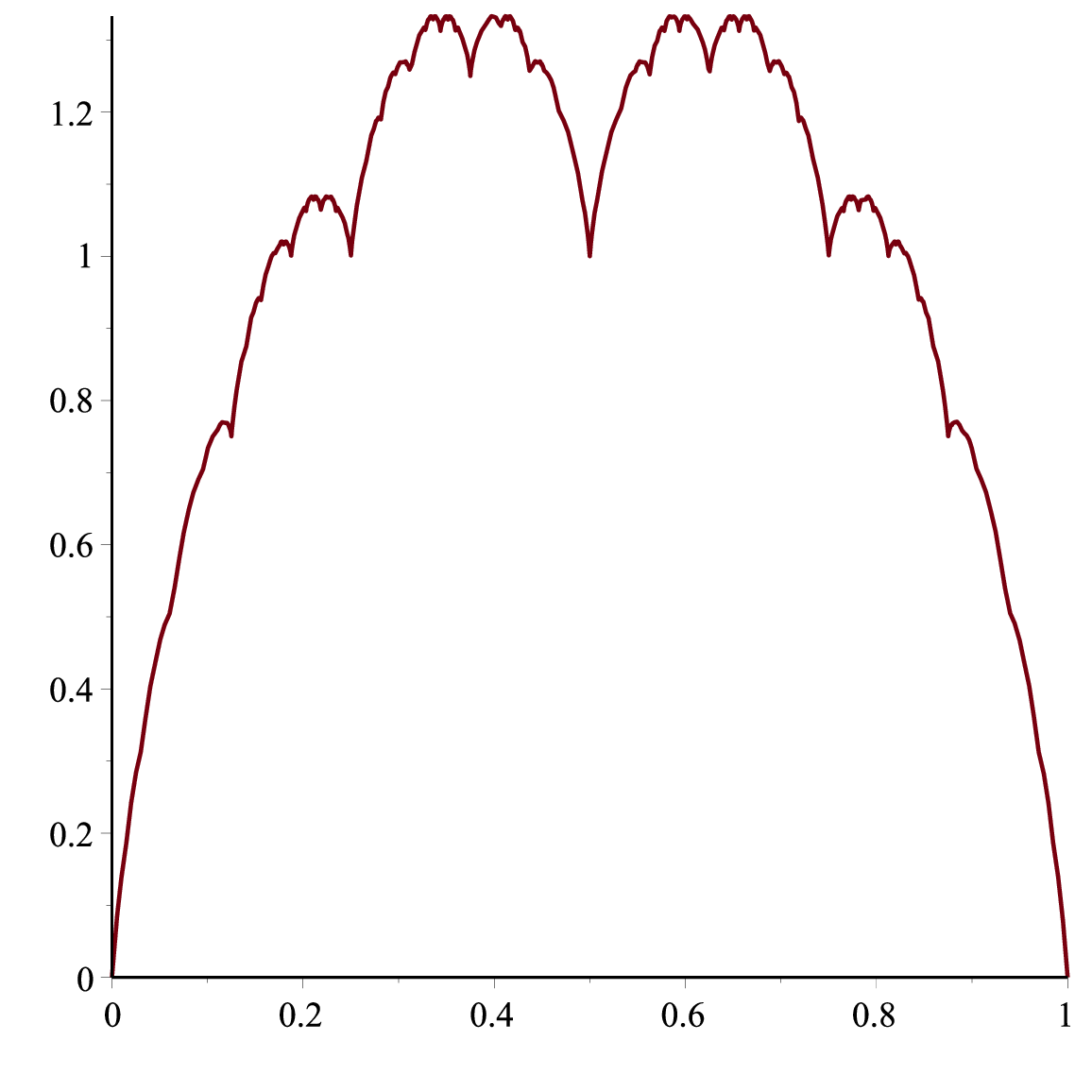}
\caption{Takagi's function $T(x)=\sum_{n\geq 0}\sum_{i=0}^{2^n-1} 2^{-n}\sigma_{n,i}(x)$ is nowhere differentiable.}
\end{figure}

\section{Function algebras}

We first prove Theorem~\ref{thm:characterization}.

Obviously, (2) is a weaker condition than (3). Now suppose that $f$ is a non-constant function, such that there exists some $x_0$ and some $\alpha>0$, such that 
\[
\limsup_{y\searrow x_0}\frac{|f(y)-f(x_0)|}{|y-x_0|^\alpha} <\infty.
\]
Put $n=\lfloor\frac{1}{\alpha}\rfloor+1$. Now consider the polynomial $P(t)=(t-f(x_0))^n$. As $\alpha$ is positive, we have $n\geq 1$, thus, $P$ is non-constant. We claim that $P(f(x))$ has a vanishing derivative at $x_0$ from the right. We have
\begin{multline*}
\limsup_{y\searrow x_0}\frac{|P(f(y))-P(f(x_0))|}{|y-x_0|} = \limsup_{y\searrow x_0}\frac{|P(f(y))|}{|y-x_0|} = 
\limsup_{y\searrow x_0}\frac{|f(y)-f(x_0)|^n}{|y-x_0|}\\
\leq\underbrace{\left(\limsup_{y\searrow x_0}\frac{|f(y)-f(x_0)|}{|y-x_0|^{\alpha}}\right)^{1/\alpha}}_{<\infty}\underbrace{\left(\limsup_{y\searrow x_0}|f(y)-f(x_0)|^{n-\frac{1}{\alpha}}\right)}_{=0} = 0,
\end{multline*}
and our claim follows. We conclude that (2) implies (1).

\begin{figure}
\includegraphics[width=\textwidth, height = 5cm]{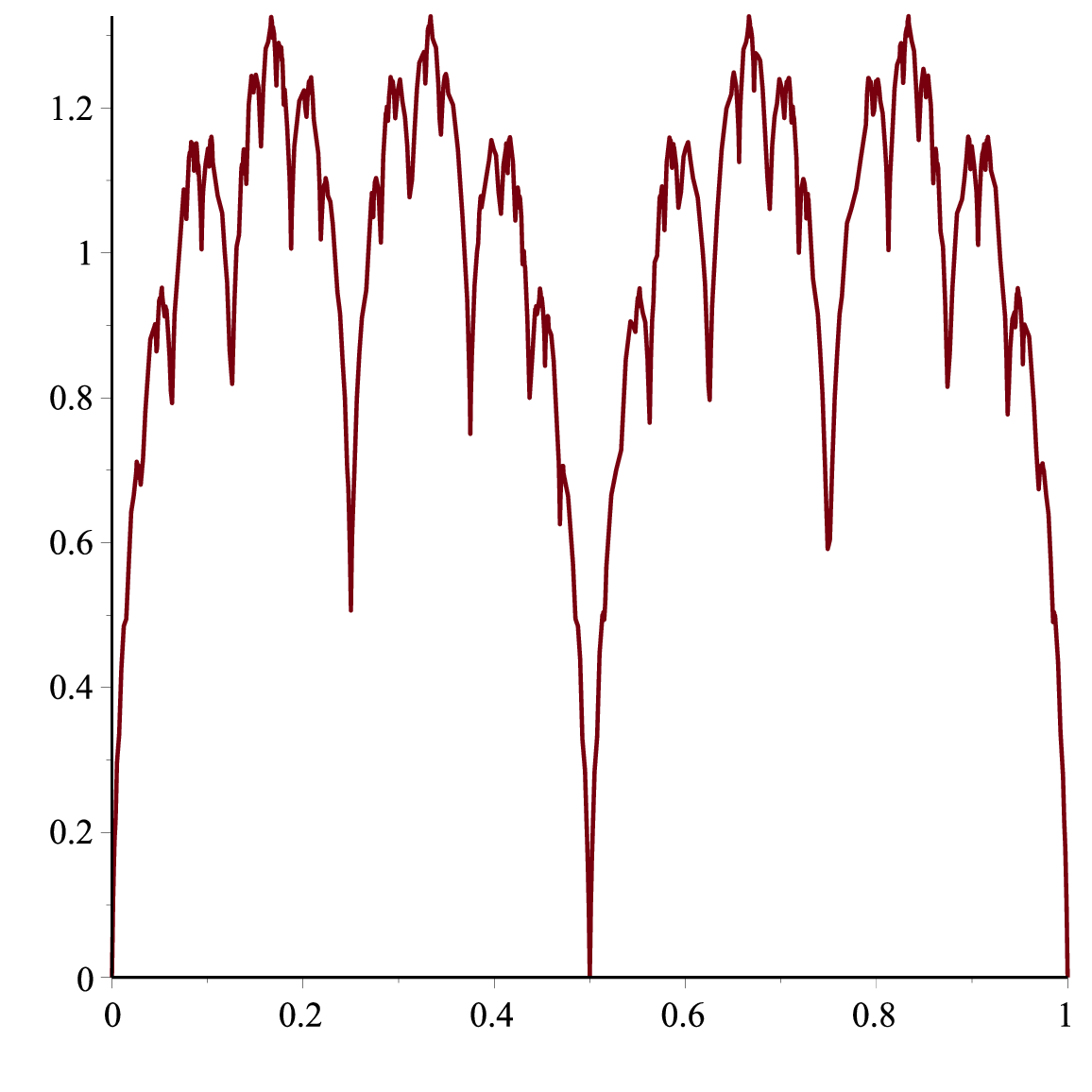}
\caption{Slightly increasing the coefficients in Takagi's function yields a function  $T_2(x)=\sum_{n\geq 0}\sum_{i=0}^{2^n-1} n2^{-n}\sigma_{n,i}(x)$ that is nowhere locally Lipschitz continuous.}
\end{figure}

It remains to show that (1) implies (3). Let $f$ be a completely non-Hölder function, $F$ a non-constant locally analytic function, $\alpha\in(0,1)$, and $x_0\in[0,1]$. As $F$ is non-constant, not all derivatives of $F$ vanish in $y_0=f(x_0)$, and we conclude that there is an integer $k$ and constants $c, \epsilon>0$, such that $|F(y)-F(y_0)|>c|y-y_0|^k$ for all $y$ with $|y-y_0|<\epsilon$. As $f$ is continuous, but not $\frac{\alpha}{k}$-Hölder continuous, we obtain 
\begin{multline*}
\limsup_{x\searrow x_0}\frac{|F(f(x))-F(f(x_0))|}{|x-x_0|^\alpha}\geq c\limsup_{x\searrow x_0}\frac{|f(x)-f(x_0)|^k}{|x-x_0|^\alpha}\\
 = c\left(\limsup_{x\searrow x_0}\frac{|f(x)-f(x_0)|}{|x-x_0|^{\alpha/k}}\right)^k = \infty,
\end{multline*}
that is, $F\circ f$ is not $\alpha$-Hölder continuous from the right in $x_0$. We conclude that $F\circ f$ is completely non-Hölder, and the proof is complete.

Note that the vector spaces of non-differentiable functions constructed by Fonf, Guraiy and Kadets \cite{FGK},Girgensohn \cite{Girgensohn}, and Bobok \cite{Bobok} are complete with respect to uniform convergence. This leads to the question whether the algebras obtained by Theorem~\ref{thm:characterization} are also complete. This is not the case, in fact, such an algebra does not exist.

\begin{Theo}
Let $\mathcal{A}\subseteq C^0([0,1])$ a closed algebra, that contains non-constant functions, and pick some $x_0\in(0,1)$. Then $\mathcal{A}$ contains a non-constant function $f$ such that $f$ is differentiable in $x_0$.
\end{Theo}
\begin{proof}
If $\mathcal{A}$ contains a non-constant function, then $\mathcal{A}$ also contains a non-constant function that vanishes at $x_0$. In fact, pick any non-constant function in $\mathcal{A}$. Then $f(x)^2-f(x_0)f(x)$ is non-constant and vanishes at $x_0$. Pick a non-constant function $f$ with $f(x_0)=0$. Now let $F$ be any continuous function. Then we have $F\circ f\in\mathcal{A}$, as $P\circ f\in\mathcal{A}$ holds for all polynomials, and polynomials are dense in the space of all continuous functions $F:f([0,1])\rightarrow\mathbb{R}$. 

Define
\[
\omega(t)=\max\{|f(x)|:|x-x_0|\leq t\} + t.
\]
$\omega$ is obviously increasing, in particular, $\omega$ is injective.
 As $f$ is continuous, $\omega$ is continuous and satisfies $\omega(0)=0$. Then the continuous function $F=\omega^{-1}$ satisfies $|(F\circ f)(x)|\leq |x-x_0|$, in particular, $(F\circ f)^2$ is an element of $\mathcal{A}$ which is differentiable in $x_0$.
\end{proof}

\begin{figure}
\includegraphics[width=\textwidth, height = 5cm]{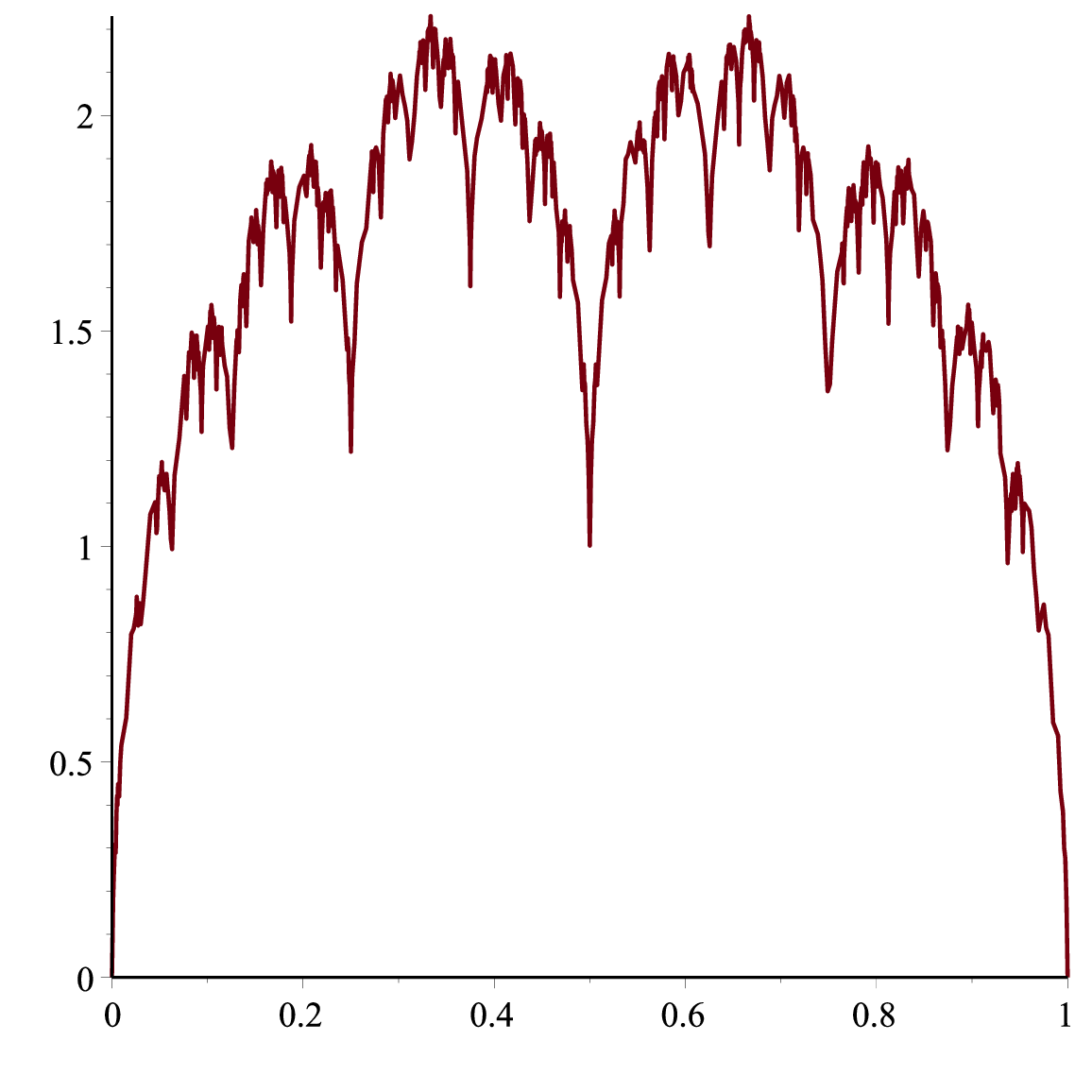}
\caption{The function $T(x)=\sum_{n\geq 0}\sum_{i=0}^{2^n-1} 2^{-n/2}\sigma_{n,i}(x)$ is $\alpha$-Hölder continuous for any $\alpha>\frac{1}{2}$ and shows similar regularity as a typical path of a Brownian motion..}
\end{figure}

\section{The set of completely non-Hölder functions}

We first prove Theorem~\ref{thm:meagre}. The proof follows the standard argument used to prove that the set of nowhere differentiable functions is meagre; therefore, we will be quite brief. Define the set
\begin{multline*}
A_{n} = \Big\{f\in C^0([0,1]): \\
\exists x_0\in[0,1], \epsilon\in\{\pm 1\}\forall y: 0<\epsilon(x_0-y)<\frac{1}{n}\Rightarrow |f(x_0)-f(y)|<|y-x_0|^{1/n}\Big\}.
\end{multline*}
The set of all functions that are not completely non-Hölder is contained in $\bigcup A_n$, hence, it suffices to show that $A_n$ is nowhere dense. Pick a function $f\in A_n$, and some $\epsilon>0$. Then $g(x)=f(x)+\epsilon \sin(100 (\epsilon/3)^{-n} x)$ has distance $\epsilon$ from $f$, and has the property that all functions $h$ which have distance $<\frac{\epsilon}{10}$ from $g$ are not contained in $A_n$. Hence $A_n$ is nowhere dense, $\bigcup A_n$ is meagre, and our claim follows.

To give an explicit example of a completely non-Hölder function we use the Faber-Schauder expansion of a continuous function. Every continuous function $f\in C([0,1])$ has a unique expansion in the form
\[
f(x) = a+bx + \sum_{n\geq 0}\sum_{i=0}^{2^n-1} \gamma_{n,i}\sigma_{n,i}(x),
\]
where the functions $\sigma_{n,i}$ form the Faber-Schauder basis introduced in the introduction.

The series on the right converges in $C^0$, if for each sequence $(i_n)$, where $i_0=0$ and $i_{n+1}\in\{2i_n, 2i_n+1\}$, the series $\sum |\gamma_{n, i_n}|$ converges. Now put $\delta_n=\min_{0\leq i\leq 2^n-1}|\gamma_{i,n}|$. Faber\cite{Faber} showed that if $\limsup\delta_n>0$, then $f$ is non-differentiable at a dense set of points. Girgensohn showed that under the same assumption, $f$ has at no point $x\in[0,1]$ a finite one-sided derivative. We now prove the following.

\begin{figure}
\includegraphics[width=\textwidth, height = 5cm]{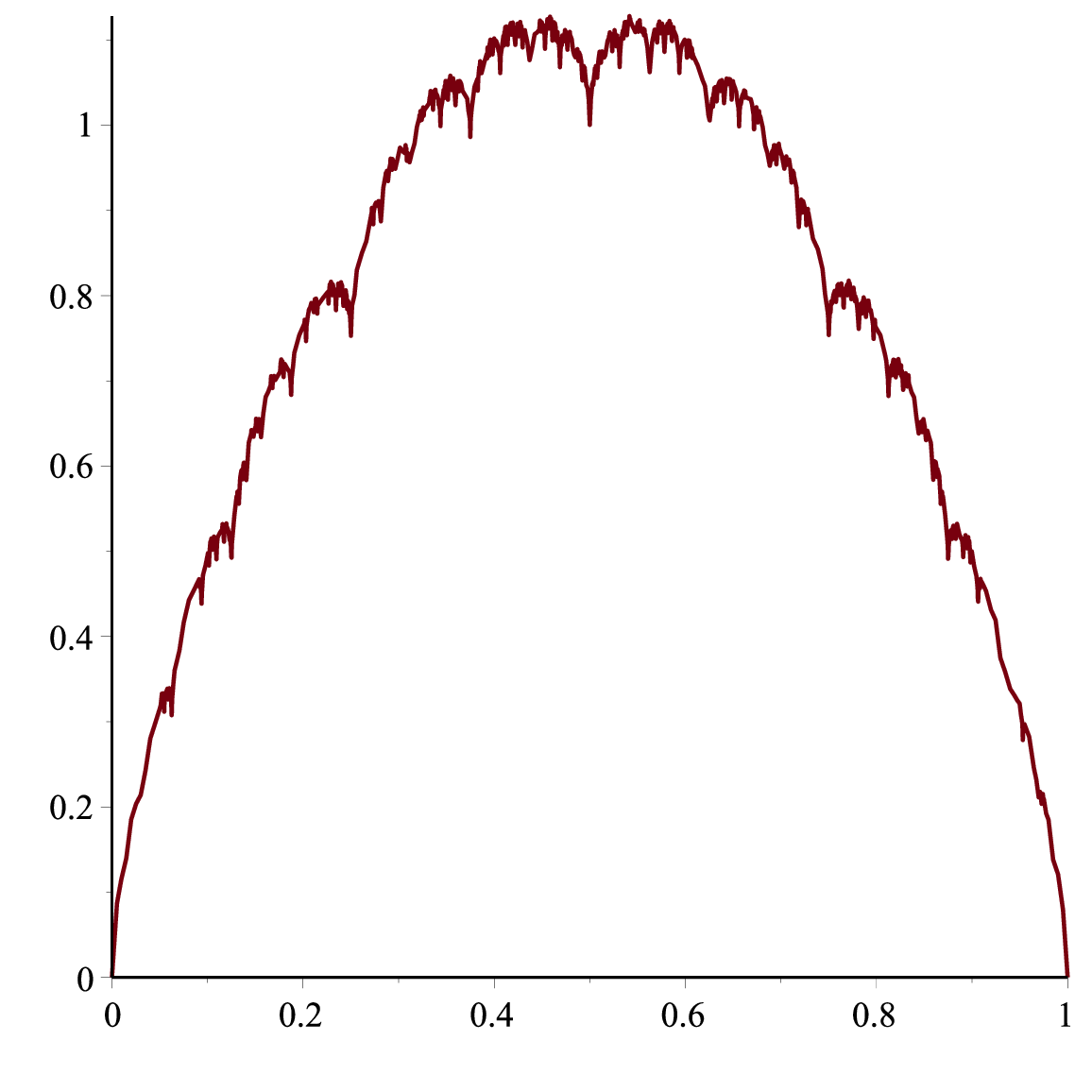}
\caption{The function $T(x)=\sum_{n\geq 0}\sum_{i=0}^{2^n-1} (1+n)^{-2}\sigma_{n,i}(x)$ is completely non-Hölder.}
\end{figure}

\begin{Theo}
\label{thm:Holder criterion}
Let $f\in C^0([0,1])$ be a function, and define
\[
\delta_n=\min_{0\leq i\leq 2^n-1}|\gamma_{i,n}|,\quad\Delta_n=\max_{0\leq i\leq 2^n-1}|\gamma_{i,n}|.
\]
Suppose that $\limsup 2^{\alpha n}\delta_n=\infty$. Then $f$ is at no point $x_0\in[0,1]$ one sided $\alpha$-Hölder continuous. If on the other hand
$\sum 2^{\alpha n}\Delta_n$
converges, then $f$ is everywhere $\alpha$-H\"older continuous.
\end{Theo}
\begin{proof}
We first show that a function with large Faber-Schauder coefficients cannot be H\"older.

Pick a constant $C$. As $\limsup 2^{\alpha n}\delta_n=\infty$, there exists an $n$, such that $\gamma_{n,a}\geq C 2^{-\alpha n}$ for all $a$. Pick such an index $n$.
Let $x_0\in[0,1]$ be arbitrary, and put $a=\lfloor 2^{n_k} x_0\rfloor+1$. We now consider the points $x_1=\frac{a}{2^{n_k+1}}$, $x_2 = \frac{a+1}{2^{n_k+1}}$. To compare $f$ at the points $x_0, x_1, x_2$, we decompose $f$ as
\[
f(x) =  \underbrace{\sum_{n<n_k}\sum_{i=0}^{2^n-1} \gamma_{n,i}\sigma_{n,i}(x)}_{=:f_1(x)} + \underbrace{\sum_{i=0}^{2^{n_k}-1} \gamma_{n,i}\sigma_{n_k,i}(x)}_{=: f_2(x)} + \underbrace{\sum_{n>n_k}\sum_{i=0}^{2^n-1} \gamma_{n,i}\sigma_{n,i}(x)}_{=: f_3(x)}.
\]
We have $f_3(x_1)=f_3(x_2)=0$. On $\left[\frac{a-1}{2^{n_k+1}}, \frac{a+1}{2^{n_k+1}}\right]$, the function $f_1$ is linear. We conclude that there is a linear function $\ell$, such that $f(x_i)=\ell(x_i)+\gamma_{n, a}\sigma_{n, a}(x_i)$ for $i=0, 1, 2$. Now
\begin{multline*}
\frac{f(x_2)-f(x_0)}{x_2-x_0} - \frac{f(x_1)-f(x_0)}{x_1-x_0} = \\
\frac{\gamma_{n, a}\sigma_{n, a}(x_2)-\gamma_{n, a}\sigma_{n, a}(x_0)}{x_2-x_0} - \frac{\gamma_{n, a}\sigma_{n, a}(x_1)-\gamma_{n, a}\sigma_{n, a}(x_0)}{x_1-x_0}  = \\
\gamma_{n,a}\left(2^{n+1} + \underbrace{\left(1-2^{n+1}(x_1-x_0)\right)}_{\geq 0}\frac{1}{x_2-x_0}\right)\geq 2^{n+1}\gamma_{n,a}>2C 2^{(1-\alpha)n}.
\end{multline*}
Hence,
\begin{multline*}
\max\big(|f(x_1)-f(x_0)|, |f(x_2)-f(x_0)|\big) \geq \\
2^{-n}\max\left(\frac{f(x_2)-f(x_0)}{x_2-x_0}, \frac{f(x_1)-f(x_0)}{x_1-x_0} \right)
>C2^{-\alpha n}>\frac{C}{2}\min(|x_2-x_0|, |x_1-x_0|)^\alpha,
\end{multline*}
and our claim follows.

For the converse direction put $\Delta_n= \max_i |\gamma_{n, i}|$. Suppose that $\sum_{n\geq 0}2^{\alpha n}\Delta_n$ converges. Then we have
\[
|f(x)-f(y)| = \left|\sum_{n\geq 0} \sum_{i=0}^{2^n-1}\gamma_{n,i} (\sigma_{n,i}(x)-\sigma_{n,i}(y))\right|.
\]
Let $n_0$ the smallest integer such that $2^{-n_0}<|x-y|$. Then we split the sum in the previous equation into terms $n\leq n_0$ and $n>n_0$. If $n\leq n_0$, we use the fact that the derivative of $\sigma_{n,i}$ is bounded by $2^n$, whereas for $n>n_0$ we use the fact that $\sigma_{n_i}$ is uniformly bounded by 1. Furthermore for each $x\in[0,1]$ and every $n$ there is exactly one $i$ such that $\sigma_{n,i}(x)\neq 0$, that is, the sum over $i$ collapses and we obtain
\begin{eqnarray*}
|f(x)-f(y)| & \leq & \sum_{n\leq n_0} 2^{n+1} \Delta_n |x-y| + \sum_{n>n_0} 2\Delta_n\\
 & = & 2|x-y| \sum_{n\leq n_0} 2^n\Delta_n + 2^{-\alpha n_0} \sum_{n>n_0} 2^{\alpha n}\Delta_n\\
 &  \leq & 2^{(1-\alpha) n_0+1}|x-y| \sum_{n\leq n_0} 2^{\alpha n}\Delta_n + 2^{-\alpha n_0} \sum_{n> n_0} 2^{\alpha n}\Delta_n\\
 & \leq & 4|x-y|^\alpha \sum_{n\geq 0} 2^{\alpha n}\Delta_n\\
  &\leq & C|x-y|^\alpha
\end{eqnarray*}
for some absolute constant $C$.
\end{proof}

\begin{Cor}
Put $f=\sum_{n\geq 0}\sum_{i=0}^{2^n-1}\frac{1}{(n+1)^2}\sigma_{n,i}$, and let $F$ be a function holomorphic in some neighbourhood of $[0, \frac{\pi^2}{6}]$. Then $F\circ f$ has no finite one-sided derivative at any point in $[0,1]$. 
\end{Cor}

By varying the coefficients we can obtain functions of any prescribed regularity, depicted in the illustrations.

\begin{Cor}
For $f\in C^0([0,1])$ define $\delta_n$ and $\Delta_n$ as above. Put $\alpha_0 = \liminf\frac{-\log\Delta_n}{n\log 2}$, $\alpha_1=\liminf\frac{-\log\delta_n}{n\log 2}$. Then $f$ is everywhere $\alpha$-H\"older for all $\alpha<\alpha_0$, and nowhere one-sided $\alpha$-H\"older for $\alpha>\alpha_1$.
\end{Cor}

\section{Algebras with few points of differentiability}

If $f$ has no point of differentiability, but is $\alpha$-H\"older for some $\alpha>0$, then the algebra generated by $f$ contains non-constant functions which are differentiable at some point. However, such points are rare. In this section we consider two examples of this phenomenon.

We first study Takagi's function which is given by the Faber-Schauder expansion
\[
T=\sum_{n\geq 0}\sum_{i=0}^{2^n-1}2^{-n} \sigma_{i,n}.
\]
Anderson and Pitt\cite{Anderson} showed that 
\[
|T(x+h)-T(x)|\leq Ch\log h^{-1}.
\]
Krüppel \cite{K1} showed that for $x=\frac{k}{2^\ell}$, $0\leq k\leq 2^\ell-1$, we have
\[
\lim_{h\rightarrow 0}\frac{T(x+h)-T(x)}{|h|\log |h|^{-1}} = 1,
\]
that is, Anderson and Pitt's estimate is best possible.

In particular Takagi's function is not Lipschitz, but $\alpha$-Hölder for every $\alpha<1$. It follows from Theorem~\ref{thm:characterization} that the algebra generated by $T$ contains functions that have a derivative at certain points. However, the functions in this algebra have very few points of differentiability.

\begin{Theo}
Let $f$ be nowhere differentiable function, which is $\alpha$-Hölder for some $\alpha>\frac{1}{2}$, $F$ a twice continuously differentiable function.  Then $F\circ f$ is differentiable in $x_0$ if and only if $F'(f(x_0))=0$. The same holds true for one-sided derivatives.
\end{Theo}
\begin{proof}
If $F'(y_0)=0$, then there exists a constant $C$ such that $|F(y)-F(y_0)|<C(y-y_0)^2$ for all $y$ sufficiently close to $y$. If $f$ is $\alpha$-Hölder, then $|F(f(x))-F(f(x_0))|<C'|x-x_0|^{2\alpha}$, which implies that $F\circ f$ has vanishing derivative at $x_0$.

Vice versa, if $F\circ f$ is differentiable in $x_0$, and $F$ has non-zero derivative at $f(x_0)$, then $f$ is differentiable at $x_0$ with derivative $\frac{(F\circ f)'(x_0)}{F'(f(x_0)}$.
\end{proof}

In particular we have the following.

\begin{Cor}
Let $f$ be a function, which is $\alpha$-H\"older for some $\alpha>\frac{1}{2}$, which does not have a one-sided derivative in any point. Let $F$ be a locally analytic function. Then $F\circ f$ has a one-sided derivative in some point $x_0$ if and only if the derivative in this point exists and equals 0.
\end{Cor}

Kahane \cite{Kahane} had shown that the level set $\{x:T(x)=\frac{2}{3}\}$ had Hausdorff dimension $\frac{1}{2}$, and
Amo, Bhouri, Carrillo, and Fern\'andez-S\'anchez \cite{Amo} showed that the level sets $\{x:T(x)=a\}$ have Hausdorff dimension $\leq\frac{1}{2}$. Hence, we obtain the following.

\begin{Cor}
Let $\mathcal{A}$ be the algebra of functions $F\circ T$, where $F$ is holomorphic in some neighbourhood of $[0, \frac{2}{3}]$, and $T$ is Takagi's function. Then a non-constant functions in $\mathcal{A}$ has a finite one-sided derivative in a set of Hausdorff dimension at most $\frac{1}{2}$, and there exist non-constant functions in this algebra which have a two sided derivative in a set of Hausdorff dimension $\frac{1}{2}$.
\end{Cor}

As a second example we consider a Brownian bridge. A Brownian bridge on $[0,1]$ is a standard Wiener process subject to $B_1=0$. It is easy to see that $B_t$ has the Faber-Schauder expansion $\sum_{n\geq 0} \sum_{0\leq i\leq 2^n-1} 2^{-{n/2}}\xi_{n,i}\sigma_{n,i}$, where the $\xi_{n,i}$ are independent variables following a normal distribution with variance 1. It follows immediately that the path of a Brownian bridge is almost surely $\alpha$-H\"older continuous for all $\alpha<\frac{1}{2}$.

The reverse estimate does not follow directly from Theorem~\ref{thm:Holder criterion}, as the minimum of the absolute value of $2^n$ independent standard normal variables is usually of magnitude $2^{-n}$, that is, $\delta_n\approx 2^{-3n/2}$. However, for each $x$ only one of the coefficients of the Faber-Schauder influences the value at $x$, that is, while there are almost vanishing coefficients, it is highly unlikely that these exceptional coefficients cluster together to produce some $x$ such that $B_t$ is differentiable at $x$. In fact, the following theorem follows from the proof of Theorem~\ref{thm:Holder criterion}.

\begin{Cor}
Identify the vertices of the infinite binary tree with $\{(n,i)| n\geq 0, 0\leq i\leq 2^n-1\}$, where $(n,i)$ is connected to $(n+1, 2i)$ and $(n+1, 2i+1)$. Let $f\in C^0([0,1])$ be a function with Faber-Schauder expansion $\sum_{n\geq 0}\sum_{0\leq i\leq 2^n-1} \gamma_{n,i}\sigma_{n,i}$. Suppose that for every branch$((n, i_n))_{n\geq 0}$ in the tree we have $\limsup 2^{\alpha n}|\gamma_{n, i_n}|=\infty$. Then $f$ is at no point one sided $\alpha$-H\"older.
\end{Cor}
The reader might wonder why we present this generalization only here and not as part of Theorem~\ref{thm:Holder criterion}. The reason is that outside a probabilistic context we do not see any natural examples where the large coefficients for different branches differ.

We can now prove the following.

\begin{Theo}
A Brownian bridge has almost surely no point where it is $\alpha$-H\"older from at least one side, provided that $\alpha>\frac{1}{2}$.
\end{Theo}
\begin{proof}
Fix $\epsilon>0$. 
Consider an infinite binary tree, and attach to each vertex a random variable, such that all random variables are independent following a normal distribution with variance 1. Call a vertex on the $n$-th level of the tree exceptional, if the absolute value of the associated random variable is less than $2^{-\epsilon n}$. If we can show that almost surely every branch contains infinitely many non-exceptional vertices, our claim follows.

Pick an integer $k>\frac{1}{\epsilon}$. Then the probabiliy that there exists a path connecting a vertex on the $n$-th layer with a vertex on the $n+k$-th layer passing only through exceptional vertices is at most $2^{n+2k-(k+1)\epsilon n}$, as there are $2^{n+k}$ such paths, the probability that a vertex is exceptional is $<2^{-\epsilon n+1}$, and the events "$v$ is exceptional" and "$w$ is exceptional" are independent. As $\sum_{n\geq 0} 2^{n-(k+1)\epsilon n}$ converges, we find by the Borel-Cantelli Lemma that almost surely there are only finitely many $n$, such that such a path exists. In particular we have almost surely that for all but finitely many $n$ we have that every branch contains a non-exceptional vertex on one of the levels $n, n+1, \ldots, n+k$. In particular, almost surely every path contains infinitely many non-exceptional vertices.
\end{proof}
Note that the same result could be proven directly using the independence and distribution of increments of a Brownian motion, however, we choose to give this proof using the Faber-Schauder expansion, as the latter can be applied to other settings, such as Brownian bridges with obstacles or solutions of stochastic differential equations.

If $B_t$ is a Brownian bridge, then the almost sure Hausdorff dimension of the set $\{x:B_t(x)=0\}$ is easily determined to be $\frac{1}{2}$. It is much more difficult to deal with all level sets simultaneously, as there are uncountably many level sets. Nevertheless Perkins\cite{Perkins} showed that we almost surely have that all level sets of a Brownian motion are either empty, consist of a single point, or have Hausdorff dimension $\frac{1}{2}$, and the proof carries over to Brownian bridges. We therefore obtain the following.

\begin{Theo}
Let $B_t(x)$ be a path of a Brownian bridge. Then we have almost surely that for every holomorphic function $F$ the set of points $x$ such that the function $F\circ B_t(x)$ has a one-sided derivative in $x$ is either empty, consists of a single point, or has Hausdorff dimension $\leq\frac{1}{2}$.
\end{Theo}

To obtain algebras consisting of functions with less points of differentiability we have to reduce the size of the level sets. We now construct a large sets of such functions.
\begin{Lem}
\label{Lem:Dimension}
Let $(n_k)$ be a strictly increasing sequence of non-negative integers such that $\frac{n_{k+1}}{n_k}\rightarrow\infty$. Then the level sets of the function
\[
f=\sum_{k\geq 1}\sum_{i=1}^{2^{n_k}-1} 2^{-n_k}\sigma_{n_k, i}
\]
 have Hausdorff dimension $0$.
\end{Lem}
\begin{proof}
Pick a real number $\lambda$.
Let $f_K$ be the partial sum
\[
f_K=\sum_{k=1}^K\sum_{i=1}^{2^{n_k}-1} 2^{-n_k}\sigma_{n_k, i}.
\]
Then $\|f-f_K\|_{\infty}\leq 2^{-n_{k+1}+1}$. The derivative of $f_K$ exists for all but finitely many points, and for all $x$ with finitely many exceptions we have $|f_K'(x)-f_{K-1}'(x)|=1$. 

Now let $x\in[\frac{a}{2^{n_K}}, \frac{a+1}{2^{n_K}}]$ be a real number such that $f(x)=\lambda$. If $f_K'(x)\neq 0$, then $|f_K(x)-\lambda|\leq 2^{-n_{K+1}+1}$, therefore $x$ is contained in a subinterval of length $\frac{2^{-n_{K+1}+1}}{|f'(x)|}\leq 2^{-n_{K+1}+1}$. If $f_K'(x)=0$, then $f_{K-1}'(x)=\pm 1$, hence, if we define the integer $b$ by means of the relation $x\in[\frac{b}{2^{n_{K-1}}}, \frac{b+1}{2^{n_{K-1}}}]$, the integer $a$ is uniquely defined by the relations
\[
a\in\left[\frac{b}{2^{n_{K-1}}}, \frac{b+1}{2^{n_{K-1}}}\right], \quad f_K'\left(\frac{a+0.5}{2^{n_K}}\right)=0,
\]
and there is a solution of the equation $f(x)=\lambda$ with $\frac{a}{2^{n_K}}\leq x\leq\frac{a+1}{2^{n_K}}$.

We conclude that if $N_K$ is the number of intervals of the form $[\frac{a}{2^{n_K}}, \frac{a+1}{2^{n_K}}]$, which contain a solution of the equation $f(x)=\lambda$, then for any $\epsilon>0$ we get
\[
N_{K+1} \leq 3N_K + 2^{n_K-n_{K+1}} N_{K-1}\leq 2^{n_K+2} = \left(2^{n_{K+1}}\right)^{\frac{(1+\frac{2}{n_K})}{n_{K+1}/n_K}}<\left(2^{n_{K+1}}\right)^\epsilon,
\]
provided that $K$ is sufficiently large.

\end{proof}
\begin{Cor}
There exists a $2^{\aleph_0}$-dimensional vector space of functions $f\in C^0([0,1])$, which are $\alpha$-H\"older for all $\alpha<1$, have no one-sided derivative in any point, and have level sets of Hausdorff dimension $0$.
\end{Cor}
\begin{proof}
Define
\[
f_\lambda=\sum_{n=1}^\infty \sum_{i=0}^{2^{2^{2^{\lfloor n^\lambda\rfloor}}}-1} 2^{-2^{\lfloor n^\lambda\rfloor}}\sigma_{2^{2^{\lfloor n^\lambda\rfloor}}, i}
\]
It is clear that $\{f_\lambda:\lambda>1\}$ is linearly independent, and as the indices of non-vanishing coefficients are a subset of $\{2^{2^n}:n\in\N\}$, we have that every linear combination of these functions satisfies the assumptions of Lemma~\ref{Lem:Dimension}. Furthermore by Girgensohn's criterion mentioned above we have that all elements of $\langle f_\lambda:\lambda>1\rangle$ have no finite one-sided derivative anywhere.
\end{proof}

Jan-Christoph Schlage-Puchta\\
Mathematisches Institut\\
Ulmenstraße 69, Haus 3\\
18057 Rostock\\
Germany\\
\verb+jan-christoph.schlage-puchta@uni-rostock.de+
\end{document}